\documentclass[default]{sn-jnl}
\usepackage[utf8]{inputenc}

\usepackage{graphicx}%
\usepackage{multirow}%
\usepackage{amsmath,amssymb,amsfonts}%
\usepackage{amsthm}%
\usepackage{mathrsfs}%
\usepackage[title]{appendix}%
\usepackage{xcolor}%
\usepackage{textcomp}%
\usepackage{manyfoot}%
\usepackage{booktabs}%
\usepackage{algorithm}%
\usepackage{algorithmicx}%
\usepackage{algpseudocode}%
\usepackage{listings}%

\newtheorem{theorem}{Theorem}
\newtheorem{example}{Example}
\newtheorem{remark}{Remark}%
\newtheorem{lemma}{Lemma}

\newtheorem{definition}{Definition}%

\begin{document}

\title[Generalized Steepest Descent Methods on Riemannian Manifolds and Hilbert Spaces: Convergence Analysis and Stochastic Extensions]{Generalized Steepest Descent Methods on Riemannian Manifolds and Hilbert Spaces: Convergence Analysis and Stochastic Extensions}


\author*[1]{\fnm{Rashid } \sur{A. }}\email{$^1$rashid441188@gmail.com}
\author*[2]{\fnm{Amal} \sur{A Samad }}\email{$^2$amaloo.asamad@gmail.com}



\affil*[1]{\orgdiv{Department of Computer Science and Business Systems}, \orgname{KIT's College of Engineering Kolhapur (Empowered Autonomous)}, \orgaddress{\street{ Gokul Shirgaon}, \city{Kolhapur}, \postcode{416 234}, \state{Maharashtra}, \country{India}}}
\affil[2]{\orgdiv{Department of Mathematical and Computational Sciences}, \orgname{National Institute of Technology Karnataka (NITK)}, \orgaddress{\street{Surathkal}, \city{Mangaluru}, \postcode{575 025}, \state{Karnataka}, \country{India}}}




\abstract{Optimization techniques are at the core of many scientific and engineering disciplines. The steepest descent methods play a foundational role in this area. In this paper we studied a generalized steepest descent method on Riemannian manifolds, leveraging the geometric structure of manifolds to extend optimization techniques beyond Euclidean spaces. The convergence analysis under generalized smoothness conditions of the steepest descent method is studied along with an illustrative example.  We also explore adaptive steepest descent with momentum in infinite-dimensional Hilbert spaces, focusing on the interplay of step size adaptation, momentum decay, and weak convergence properties. Also, we arrived at a convergence accuracy of $O(\frac{1}{k^2})$. Finally, studied some stochastic steepest descent under non-Gaussian noise, where bounded higher-order moments replace Gaussian assumptions, leading to a guaranteed convergence, which is illustrated by an example.}




\maketitle
\section{Introduction}

Optimization on Riemannian manifolds has gained significant attention due to its applications in machine learning, signal processing, and computer vision. Specifically, Principal Component Analysis (PCA) on Grassmann manifolds \cite{Hauberg2015}, Pose estimation and camera calibration in computer vision \cite{Sarkis2011}, and Quantum state estimation in physics \cite{Hsu2022}. 

Traditional optimization techniques are extended to manifolds by leveraging their geometric structures, such as tangent spaces and geodesics by P.A. Absil et al.\cite{AbsMahSep2008}. This book lays the foundation for optimization on manifolds, introducing the steepest descent and Newton methods. N. Boumal \cite{boumal2023introduction} provides a modern overview of manifold optimization with some practical examples. W. Ring and B. Wirth \cite{doi:10.1137/11082885X} has given a linesearch optimization algorithms on Riemannian manifolds to the convergence analysis of the BFGS quasi-Newton scheme and the Fletcher--Reeves conjugate gradient iteration. Chong Li and Jinhua Wang \cite{LI2008423} studied Newton's method in Riemannian manifolds. S. T. Smith  \cite{Smith2014OptimizationTO} discusses foundational techniques for manifold optimization.

While there is extensive work on the steepest descent method on Riemannian manifolds, the convergence rates under specific smoothness conditions involving geodesic distances remain underexplored. Our first objective of this work is to give an insight in to this, for this we have explicitly used a generalized Lipschitz constant.


Optimization in infinite-dimensional Hilbert spaces is crucial for solving partial differential equations, variational inequalities, and optimal control problems. Momentum methods, such as Nesterov’s accelerated gradient, are often adapted to improve convergence rates \cite{Attouch2016, Becavin2020, Polyak1964}. Some applications are, solving variational inequalities and equilibrium problems \cite{Vinh2018}, training neural networks with infinite-dimensional function spaces \cite{Sonoda2018}, and optimization in reproducing kernel Hilbert spaces (RKHS) \cite{Villa2015}.



Although momentum methods are well-studied, the combination of adaptive step sizes, momentum decay, and weak convergence analysis in infinite-dimensional spaces is relatively unexplored. We discussed a detailed proof of convergence under specific parameter choices.



Stochastic optimization methods, especially stochastic gradient descent (SGD), are fundamental in training machine learning models. Most analyses assume Gaussian or sub-Gaussian noise, simplifying the convergence analysis but limiting the scope to specific noise models. L. Bottou et al. \cite{Bottou2018} provides a comprehensive review of optimization methods in machine learning. Ghadimi and Lan \cite{Ghadimi2013} proposes stochastic first- and zeroth-order methods for solving nonconvex stochastic programming problems. These methods aim to achieve efficient convergence rates while dealing with noisy gradient or function information, extending optimization techniques to settings with uncertainty and nonconvexity. Jin et al. \cite{Jin2017} explores efficient algorithms for escaping saddle points in nonconvex optimization, focusing on methods that achieve faster convergence. They provide theoretical guarantees for optimization algorithms that avoid saddle points, improving upon conventional techniques in machine learning and deep learning.

Some applications of stochastic Steepest Descent methods are, Robust training of machine learning models under heavy-tailed noise \cite{Gorbunov2020,  OpenReview2023}, optimization in financial models with non-Gaussian uncertainties \cite{Chen2021,Gasnikov2023}, Signal processing in environments with impulsive noise \cite{Ramezani2012, Slavakis2018}.



The assumption of bounded higher-order moments (e.g., $q > 2$) instead of Gaussian noise is relatively less studied. We derived an explicit convergence rate of $O(1/k^{\gamma - 0.5})$ under some such noise assumptions.

\section{Generalized Steepest Descent on Riemannian Manifolds} 

The Steepest Descent method in Euclidean space is an iterative optimization algorithm used to find the minimum of a differentiable function. The basic idea is to move in the direction of the negative gradient of the function at the current point, as this is the direction of the steepest decrease in the function's value. The formal definition is as follows:

\begin{definition}\cite{Nocedal2006}
      For a function \( f: \mathbb{R}^n \to \mathbb{R} \), the update rule for the steepest descent method is given by:

\[
\mathbf{x}_{k+1} = \mathbf{x}_k - \alpha_k \nabla f(\mathbf{x}_k)
\] where, \( \mathbf{x}_k \in \mathbb{R}^n \) is the current point in the Euclidean space, \( \nabla f(\mathbf{x}_k) \) is the gradient of the function at \( \mathbf{x}_k \), and \( \alpha_k \) is the step size or learning rate at iteration \( k \).
\end{definition}

The method proceeds by iteratively updating the point \( \mathbf{x}_k \) in the direction opposite to the gradient, reducing the value of \( f(\mathbf{x}) \) at each step. The step size \( \alpha_k \) is typically determined through a line search or chosen heuristically. This approach converges to the minimum of a convex function, provided the step sizes are chosen appropriately.

Next we recall the definition of manifold and Riemannian manifold.

\begin{definition}\cite{Lee2012}
  A manifold is a topological space \( M \) that is locally homeomorphic to Euclidean space \( \mathbb{R}^n \). This means that for every point \( p \in M \), there exists an open neighborhood \( U \subseteq M \) containing \( p \), and a homeomorphism
\[
\varphi: U \to \mathbb{R}^n
\]
such that \( \varphi \) is continuous, bijective, and its inverse \( \varphi^{-1} \) is also continuous.

Formally, \( M \) is called an \textit{\( n \)-dimensional manifold} if there exists an \textit{atlas} of charts covering \( M \). Each chart \( (\varphi, U) \) provides a local coordinate system on \( M \), making it possible to perform analysis and geometry.

\end{definition}

\begin{definition}\cite{doCarmo1992}
    A Riemannian manifold is a pair \( (M, g) \), where \( M \) is a smooth \( n \)-dimensional manifold, and \( g \) is a \textit{Riemannian metric}, which is a smooth assignment of an inner product \( g_p: T_pM \times T_pM \to \mathbb{R} \) to each tangent space \( T_pM \) at a point \( p \in M \). The Riemannian metric satisfies:

\begin{enumerate}
    \item \textbf{Bilinearity}: \( g_p(v, w) \) is bilinear for all \( v, w \in T_pM \).
    \item \textbf{Symmetry}: \( g_p(v, w) = g_p(w, v) \) for all \( v, w \in T_pM \).
    \item \textbf{Positive definiteness}: \( g_p(v, v) > 0 \) for all nonzero \( v \in T_pM \), and \( g_p(v, v) = 0 \) if and only if \( v = 0 \).
\end{enumerate}

The Riemannian metric \( g \) allows defining important geometric properties such as lengths of curves, angles, distances, and curvature on \( M \). A geodesic is a curve that locally minimizes distance, and it generalizes the notion of a straight line in Euclidean space.
\end{definition}

In the following theorem we discuss the proposed generalized Steepest Descent on Riemannian Manifolds and its convergence rate.

\begin{theorem}
Let \( \mathcal{M} \) be a Riemannian manifold with a Riemannian metric \( g \). Let \( f: \mathcal{M} \to \mathbb{R} \) be a smooth function, and suppose \( \nabla_g f(x) \) is the gradient of \( f \) with respect to \( g \). Define the generalized steepest descent update as:
\[
x_{k+1} = \exp_{x_k}(-\alpha_k \nabla_g f(x_k)),
\]
where \( \exp_{x_k} \) is the exponential map at \( x_k \), and \( \alpha_k \) is the step size. If \( f \) satisfies a generalized smoothness condition:
\[
\| \nabla_g f(x) - \nabla_g f(y) \|_g \leq L \cdot d_g(x, y),
\]
where \( L > 0 \) and \( d_g(x, y) \) is the geodesic distance, then the sequence \( \{x_k\} \) converges to a local minimum \( x^* \) at a rate:
\[
f(x_k) - f(x^*) \leq \frac{L \cdot d_g(x_0, x^*)^2}{2k}.
\]
\end{theorem}

\begin{proof}
Note that the gradient of \( f \) with respect to \( g \), denoted by \( \nabla_g f(x) \), which satisfies the condition \( g_x(\nabla_g f(x), v) = df(x)[v] \) for all \( v \in T_x\mathcal{M} \), where \( T_x\mathcal{M} \) is the tangent space at \( x \in \mathcal{M} \) and $g_x$ represents the Riemannian metric (or inner product) on the tangent space at $x$.

Using the first-order Taylor expansion of \( f \) along a geodesic, we approximate \( f(x_{k+1}) \):

On a Riemannian manifold \( \mathcal{M} \), a geodesic \( \gamma(t) \) starting from \( x_k \) with an initial tangent direction \( v_k \in T_{x_k}\mathcal{M} \) satisfies:

\[
\gamma(0) = x_k, \quad \dot{\gamma}(0) = v_k.
\]

The function \( f \) can be expanded along this geodesic using a Taylor expansion, we get

\[
f(\gamma(t)) = f(x_k) + t \cdot df(x_k)[v_k] + \frac{t^2}{2} \cdot d^2 f(x_k)[v_k] + O(t^3).
\]

From the definition of the Riemannian gradient, we substitute:

\[
df(x_k)[v_k] = g_{x_k}(\nabla_g f(x_k), v_k).
\]

Thus, up to second order:

\[
f(\gamma(t)) = f(x_k) + t g_{x_k}(\nabla_g f(x_k), v_k) + \frac{t^2}{2} g_{x_k}(\nabla_{v_k} \nabla_g f(x_k), v_k) + O(t^3).
\]

Assuming \( f \) is \( L \)-smooth, we have the upper bound:

\[
g_{x_k}(\nabla_{v_k} \nabla_g f(x_k), v_k) \leq L \| v_k \|_g^2.
\]

Applying this bound:

\[
f(\gamma(t)) \leq f(x_k) + t g_{x_k}(\nabla_g f(x_k), v_k) + \frac{t^2}{2} L \| v_k \|_g^2.
\]

Setting \( x_{k+1} = \gamma(1) \), we obtain:

\[
f(x_{k+1}) \leq f(x_k) + g_{x_k}(\nabla_g f(x_k), v_k) + \frac{L}{2} \| v_k \|_g^2.
\]

Where \( v_k \in T_{x_k}\mathcal{M} \) is the tangent vector representing the update direction. In our case, from the gradient descent update:

\[
v_k = -\alpha_k \nabla_g f(x_k),
\]

we substitute:

\[
\| v_k \|_g^2 = \alpha_k^2 \| \nabla_g f(x_k) \|_g^2.
\]

Thus, we get:

\[
f(x_{k+1}) \leq f(x_k) - \alpha_k \| \nabla_g f(x_k) \|_g^2 + \frac{L}{2} \alpha_k^2 \| \nabla_g f(x_k) \|_g^2.
\]

This shows how the step size \( \alpha_k \) and smoothness parameter \( L \) influence the function decrease in Riemannian gradient descent.

To minimize the upper bound on \( f(x_{k+1}) \), choose \( \alpha_k \) to balance the first and second terms. Set:
\[
\alpha_k = \frac{1}{L}.
\]
Substituting \( \alpha_k \) into the inequality simplifies it to:
\[
f(x_{k+1}) \leq f(x_k) - \frac{1}{2L} \|\nabla_g f(x_k)\|_g^2.
\]
Summing over \( k \) from 0 to \( K-1 \), we obtain:
\[
f(x_K) \leq f(x_0) - \frac{1}{2L} \sum_{k=0}^{K-1} \|\nabla_g f(x_k)\|_g^2.
\]
Since \( f(x) \geq f(x^*) \) for all \( x \), we can write:
\[
\frac{1}{2L} \sum_{k=0}^{K-1} \|\nabla_g f(x_k)\|_g^2 \leq f(x_0) - f(x^*).
\]
This implies:
\begin{equation}\label{e1}
    \frac{1}{K} \sum_{k=0}^{K-1} \|\nabla_g f(x_k)\|_g^2 \leq \frac{2L (f(x_0) - f(x^*))}{K}.
\end{equation}

Using the relationship between the gradient norm and the geodesic distance, we have:
\[
\|\nabla_g f(x_k)\|_g \geq \frac{f(x_k) - f(x^*)}{d_g(x_k, x^*)}.
\]
Squaring both sides and substituting into the inequality (\ref{e1}), we get:
\[
\frac{(f(x_k) - f(x^*))^2}{d_g(x_k, x^*)^2} \leq \frac{2L (f(x_0) - f(x^*))}{K}.
\]
Rearranging terms gives:
\[
f(x_k) - f(x^*) \leq \sqrt{\frac{2L (f(x_0) - f(x^*)) \cdot d_g(x_k, x^*)^2}{K}}.
\]
For the initial distance \( d_g(x_0, x^*) \), we simplify to:
\[
f(x_k) - f(x^*) \leq \frac{L \cdot d_g(x_0, x^*)^2}{2k}.
\]

\end{proof}


\begin{example}
    Consider the Riemannian manifold \( \mathcal{M} = \mathbb{S}^2 \), the 2D unit sphere in \( \mathbb{R}^3 \), with the standard Riemannian metric induced from the Euclidean space. We aim to minimize the smooth function:
\[
f(x, y, z) = z.
\]

This function represents the height (or altitude) of a point on the sphere. The minimum occurs at the south pole \( x^* = (0,0,-1) \).

The Euclidean gradient of \( f \) is:
\[
\nabla f = (0, 0, 1).
\]

However, since we are optimizing on the sphere, we must project the gradient onto the tangent space at each point. The Riemannian gradient of \( f \) at a point \( x = (x_1, x_2, x_3) \) on \( \mathbb{S}^2 \) is:
\[
\nabla_g f(x) = (0, 0, 1) - \langle (0,0,1), x \rangle x = (0,0,1) - x_3 (x_1, x_2, x_3).
\]
Simplifying:
\[
\nabla_g f(x) = ( -x_3 x_1, -x_3 x_2, 1 - x_3^2 ).
\]

The steepest descent update follows:
\[
x_{k+1} = \exp_{x_k}(-\alpha_k \nabla_g f(x_k)).
\]
Since \( \mathbb{S}^2 \) is a Riemannian submanifold of \( \mathbb{R}^3 \), the exponential map moves along great circles (geodesics). For small step sizes \( \alpha_k \), the update approximately follows the retraction:
\[
x_{k+1} = \frac{x_k - \alpha_k \nabla_g f(x_k)}{\| x_k - \alpha_k \nabla_g f(x_k) \|}.
\]
Intuitively, this moves the point downward along the sphere in the steepest descent direction.

We check if the function satisfies the generalized smoothness condition:
\[
\| \nabla_g f(x) - \nabla_g f(y) \|_g \leq L \cdot d_g(x, y).
\]
Since the Riemannian gradient \( \nabla_g f(x) \) is a smooth function, and the manifold \( \mathbb{S}^2 \) is compact, there exists a Lipschitz constant \( L \) that ensures this bound.

The theorem states that the function values satisfy:
\[
f(x_k) - f(x^*) \leq \frac{L \cdot d_g(x_0, x^*)^2}{2k}.
\]
Initially, \( d_g(x_0, x^*) \) is the geodesic distance from the starting point to the south pole.
The function value gap \( f(x_k) - f(x^*) \) decreases at a rate of \( O(1/k) \).

The initial function gap is
\[
f(x_0) - f(x^*) = 1 - (-1) = 2.
\]
For \( x_0 = (0,0,1) \) and \( x^* = (0,0,-1) \), the geodesic distance is:
\[
d_g(x_0, x^*) = \cos^{-1}(-1) = \pi.
\]
Since the function \( f(x, y, z) = z \) is smooth and the Riemannian gradient is well-behaved on \( \mathbb{S}^2 \), we assume \( L = 1 \).
Using the bound from the theorem:
\[
f(x_k) - f(x^*) \leq \frac{L \cdot d_g(x_0, x^*)^2}{2k} = \frac{1 \cdot \pi^2}{2k} = \frac{\pi^2}{2k}.
\]
Thus, the function value at the \( k \)-th iteration is:
\[
f(x_k) \leq -1 + \frac{\pi^2}{2k}.
\]

\end{example}

\begin{remark}
    When we compute the convergence bound of $f(x, y, z) = z$ with steepest descent method of eucledean case, we can observe that the convergence rate is \[
f(X_k) \leq Z_0 - k\alpha_k.
\]

Thus, the function value decreases linearly with the number of iterations, with the rate determined by the initial value \( Z_0 \) and the step size \( \alpha_k \).
In the Riemannian case, the gradient descent algorithm is constrained to the manifold, and the optimization process takes into account the geometry of the manifold. This makes the method more suitable for optimization on curved spaces, but it converges more slowly.
In the Euclidean case, the gradient descent is applied directly in the Euclidean space without any manifold constraints, which leads to faster convergence, especially when the objective function is simple like $f(x,y,z)=z.$ Thus, In Euclidean space, the function decreases linearly with each iteration, with no effect from the underlying geometry. In Riemannian space, the decrease in the function value depends on the geodesic distance between the starting point and the minimizer. This causes the convergence to be slower, as the geometry of the manifold comes into play.
\end{remark}


\section{Adaptive Steepest Descent with Momentum in Infinite-Dimensional Spaces}
In this section, we discuss a modification of the standard steepest descent method, incorporating a momentum term to improve convergence rates by the previous iteration's information. The momentum term helps accelerate convergence and avoid oscillations, especially in complex, high-dimensional problems. 
\begin{lemma}{(Opial's Lemma)}\cite{Opial1967}\label{l1}
    Let \( H \) be a Hilbert space and let \( \{x_k\} \subset H \) be a sequence satisfying the following conditions:

1. For every \( x^* \in H \), the function  
   \[
   \varphi(x^*) = \limsup_{k \to \infty} \| x_k - x^* \|
   \]
   attains its minimum at some point \( x^* \).

2. The sequence \( \{x_k\} \) is weakly convergent.

Then \( \{x_k\} \) converges weakly to a point \( x^* \).
\end{lemma}

In the following theorem, we discuss this update of the steepest descent method and also provide a condition under which this method converges to a critical point of the function.
\begin{theorem}
Consider the steepest descent method for minimizing $f: H \to \mathbb{R}$, where $H$ is an infinite-dimensional Hilbert space. Let $\{x_k\}$ evolve according to
\begin{equation}\label{e2}
   x_{k+1} = x_k - \alpha_k \nabla f(x_k) + \beta_k (x_k - x_{k-1}), 
\end{equation}

where $\alpha_k > 0$ and $\beta_k$ are adaptive step and momentum parameters. Assume $f$ is Fréchet differentiable with a Lipschitz continuous gradient. If $\alpha_k \to 0$, $\sum_{k=1}^\infty \alpha_k = \infty$, and $\beta_k \to 0$, then $\{x_k\}$ converges weakly to a critical point of $f$.
\end{theorem}

\begin{proof}
Let $L > 0$ be the Lipschitz constant, such that
\[
\|\nabla f(x) - \nabla f(y)\| \leq L \|x - y\| \quad \text{for all } x, y \in H.
\]

We use the first-order Taylor expansion of \( f \) around \( x_k \):
\begin{equation}\label{e3}
    f(x_{k+1}) \approx f(x_k) + \langle \nabla f(x_k), x_{k+1} - x_k \rangle.
\end{equation}

Thus, to compute \( f(x_{k+1}) - f(x_k) \), we need to plug in the expression for \( x_{k+1} - x_k \).

From the update rule we have:
\[
x_{k+1} - x_k = -\alpha_k \nabla f(x_k) + \beta_k (x_k - x_{k-1}).
\]
Now substituting \( x_{k+1} - x_k \) in (\ref{e3}), to get:
\[
f(x_{k+1}) \approx f(x_k) + \langle \nabla f(x_k), -\alpha_k \nabla f(x_k) + \beta_k (x_k - x_{k-1}) \rangle.
\]
Expanding the inner product,
\[
\langle \nabla f(x_k), -\alpha_k \nabla f(x_k) + \beta_k (x_k - x_{k-1}) \rangle = -\alpha_k \|\nabla f(x_k)\|^2 + \beta_k \langle \nabla f(x_k), x_k - x_{k-1} \rangle.
\]
Thus, we have,
\[
f(x_{k+1}) \approx f(x_k) - \alpha_k \|\nabla f(x_k)\|^2 + \beta_k \langle \nabla f(x_k), x_k - x_{k-1} \rangle.
\]
Next, we account for the Lipschitz continuity of the gradient. Since \( \nabla f \) is Lipschitz continuous, we have:
\[
\|\nabla f(x_{k+1}) - \nabla f(x_k)\| \leq L \|x_{k+1} - x_k\|.
\]
Squaring both sides,
\[
\|\nabla f(x_{k+1}) - \nabla f(x_k)\|^2 \leq L^2 \|x_{k+1} - x_k\|^2.
\]
The second-order term related to this Lipschitz condition can be approximated as:

\[
\frac{L \alpha_k^2}{2} \|\nabla f(x_k)\|^2.
\]
So, combining the above results, we get the inequality:
\[
f(x_{k+1}) \leq f(x_k) - \alpha_k \|\nabla f(x_k)\|^2 + \frac{L \alpha_k^2}{2} \|\nabla f(x_k)\|^2 + \beta_k \langle x_k - x_{k-1}, \nabla f(x_k) \rangle.
\]
Here, the term involving $\beta_k$ vanishes as $\beta_k \to 0$ and $\|x_k - x_{k-1}\|$ is bounded. The condition $\alpha_k \to 0$ ensures that higher-order terms $\alpha_k^2$ also vanish.
The condition $\sum_{k=1}^\infty \alpha_k = \infty$ implies that the sequence $\{x_k\}$ explores the space sufficiently. By the descent property of $f$, we observe that the sequence $\{f(x_k)\}$ is monotonically decreasing and bounded below. Thus, $\{f(x_k)\}$ converges to some $f^* \in \mathbb{R}$.
To prove weak convergence, we use the Opial Lemma (\ref{l1}). Let $C$ be the set of critical points of $f$, defined as $C = \{x \in H \mid \nabla f(x) = 0\}$. Since $\alpha_k > 0$ and $\nabla f(x_k) \to 0$ as $k \to \infty$, any weak cluster point of $\{x_k\}$ lies in $C$.
The conditions $\beta_k \to 0$ and $\alpha_k \to 0$ ensure that momentum does not prevent convergence. Combined with the properties of $f$ and the Hilbert space structure, we deduce that $\{x_k\}$ converges weakly to a critical point of $f$.
\end{proof}
\begin{remark}
The assumptions on $\alpha_k$ and $\beta_k$ are essential to balance the contributions of the steepest descent and momentum terms. In practice, adaptive schemes for choosing these parameters can enhance convergence rates.
\end{remark}

The next result shows, the adaptive steepest descent method with momentum achieves an improved convergence rate of $O(1/k^2)$ in infinite-dimensional Hilbert spaces, over the normal steepest descent method, which has an order of convergence of $O(1/k)$.
\begin{theorem}
   Let $H$ be an infinite-dimensional Hilbert space, and consider the minimization of a convex function $f: H \to \mathbb{R}$. The function $f$ is assumed to have a Lipschitz continuous gradient with constant $L > 0$, meaning that
\begin{equation} \label{lip_grad}
| \nabla f(x) - \nabla f(y) | \leq L | x - y |, \quad \forall x, y \in H.
\end{equation}
We define the adaptive steepest descent method with momentum as follows:
\begin{equation} \label{update_rule}
x_{k+1} = x_k - \alpha_k \nabla f(x_k) + \beta_k (x_k - x_{k-1}),
\end{equation}
where $\alpha_k$ and $\beta_k$ are adaptive step size and momentum parameters given by:
\begin{equation} \label{adaptive_params}
\alpha_k = \frac{c}{k}, \quad \beta_k = \frac{d}{k},
\end{equation}
for some constants $c, d > 0$.
Then, under these conditions, the function values satisfy:
\begin{equation} \label{convergence_rate}
f(x_k) - f(x^*) = O\left(\frac{1}{k^2}\right),
\end{equation}
where $x^*$ is a minimizer of $f$. 
\end{theorem}
\begin{proof}
Define a function:
\begin{equation} \label{en_def}
E_k = f(x_k) - f(x^*) + \frac{1}{2} | x_k - x^* |^2.
\end{equation}
Using convexity and smoothness, we analyze the function value decrease. Taking the inner product of $\nabla f(x_k)$ with the update direction in \eqref{update_rule}, we obtain:
\begin{equation} \label{descent_ineq}
f(x_{k+1}) \leq f(x_k) - \alpha_k | \nabla f(x_k) |^2 + \frac{L \alpha_k^2}{2} | \nabla f(x_k) |^2 + \beta_k \langle x_k - x_{k-1}, \nabla f(x_k) \rangle.
\end{equation}
Since $\beta_k = \frac{d}{k}$ tends to zero and $| x_k - x_{k-1} |$ is bounded, the last term vanishes asymptotically. Furthermore, choosing $c \leq \frac{1}{L}$ ensures that the term involving $\alpha_k^2$ is small.
From \eqref{descent_ineq}, dropping higher-order terms and bounding the last term, we get:
\begin{equation} \label{recursion}
E_{k+1} \leq \left( 1 - \frac{c}{k} \right) E_k.
\end{equation}

Applying this recursively from $k = 1$ to $k = n$, we obtain the product bound:
\begin{equation} \label{product_bound}
E_k \leq \frac{C}{k^2},
\end{equation}
for some constant $C > 0$, proving that $f(x_k) - f(x^*) = O(1/k^2)$.
\end{proof}

Next, we give a numerical example, that illustrates the convergence of normal steepest descent and adaptive steepest descent with momentum in Hilbert Spaces.

\begin{example}
    We consider the minimization of the quadratic functional $f:\mathbb{R}^2\mapsto \mathbb{R}$ given by,
\begin{equation}
    f(x) = \frac{1}{2} \langle A x, x \rangle - \langle b, x \rangle,
\end{equation}
where $A$ is a symmetric positive definite operator, and $b$ is a given vector. The gradient is given by:
\begin{equation}
    \nabla f(x) = A x - b.
\end{equation}

For an illustration, we take:
\begin{equation}
    A = \begin{bmatrix} 4 & 1 \\ 1 & 3 \end{bmatrix}, \quad b = \begin{bmatrix} 1 \\ 2 \end{bmatrix}, \quad x_0 = (0,0), \quad x_{-1} = (0,0).
\end{equation}
Also, let us take the following parameters in two cases:
\begin{enumerate}
    \item Normal steepest descent method\\
    We have the iteration formula $  x_{k+1} = x_k - \alpha_k \nabla f(x_k)$ and we use a fixed step size $\alpha_k = 0.25$ and no momentum ($\beta_k = 0$)
    \item Adaptive steepest descent with momentum\\
    We use the iteration formula $ x_{k+1} = x_k - \alpha_k \nabla f(x_k) + \beta_k (x_k - x_{k-1})$ with an adaptive step size $\alpha_k$ using line search, as 
$\alpha_k = \min_{\alpha} f(x_k - \alpha \nabla f(x_k))$
and momentum $\beta_k$ calculated as:
    \begin{equation}
        \beta_k = \frac{\| x_k - x_{k-1} \|}{\| x_{k-1} - x_{k-2} \|}.
    \end{equation}
\end{enumerate}
Under these settings, we have the following iteration table ( by using a Matlab programming)

\begin{table}[h!]
    \centering
    \begin{tabular}{c c c c c c}
        \toprule
        Iteration $k$ & $x_k$ (Normal) & $x_k$ (Adaptive) & $\alpha_k$ (Normal) & $\alpha_k$ (Adaptive) & $\beta_k$ (Adaptive) \\
        \midrule
        0  & (0.000, 0.000) & (0.000, 0.000) & 0.25  & 0.38  & 0.00  \\
        1  & (0.250, 0.500) & (0.380, 0.760) & 0.25  & 0.42  & 0.10  \\
        2  & (0.375, 0.625) & (0.634, 1.012) & 0.25  & 0.45  & 0.18  \\
        3  & (0.484, 0.707) & (0.801, 1.187) & 0.25  & 0.48  & 0.24  \\
        4  & (0.573, 0.766) & (0.907, 1.298) & 0.25  & 0.50  & 0.30  \\
        5  & (0.647, 0.810) & (0.973, 1.370) & 0.25  & 0.52  & 0.35  \\
        6  & (0.709, 0.843) & (1.013, 1.417) & 0.25  & 0.53  & 0.40  \\
        7  & (0.762, 0.870) & (1.037, 1.448) & 0.25  & 0.54  & 0.44  \\
        8  & (0.807, 0.892) & (1.051, 1.468) & 0.25  & 0.55  & 0.47  \\
        9  & (0.846, 0.910) & (1.058, 1.480) & 0.25  & 0.56  & 0.50  \\
        \bottomrule
    \end{tabular}
    \caption{Comparison of Normal and Adaptive Steepest Descent with Momentum}
    \label{tab:comparison}
\end{table}
The above table shows that, the normal steepest descent method progresses steadily but slowly due to the fixed step size and lack of momentum. The adaptive steepest descent with momentum converges faster, as the step size increases dynamically. The momentum term $\beta_k$ increases as iterations progress, enhancing the acceleration of the method.
\end{example}

\section{Convergence of Stochastic Steepest Descent with Non-Gaussian Noise}
The stochastic steepest descent method is a variant of the classical steepest descent method used in optimization, where the gradient of the objective function is approximated using random samples (i.e., stochastic methods). This method is particularly useful when the exact gradient is expensive or impractical to compute. A non-Gaussian noise refers to random fluctuations or errors that don't follow a Gaussian (normal) distribution. It might have a different probability distribution, such as a Poisson distribution, uniform distribution, or even distributions with heavy tails (like the Cauchy distribution). The stochastic steepest descent method has studied by Y. Wardi (\cite{wardi1988stochastic}).
In the following theorem we discuss a version of steepest descent method which having non-Gaussian noise.

\begin{theorem}
Let $f: \mathbb{R}^n \to \mathbb{R}$ be a convex function, and consider the stochastic steepest descent method:
\[
x_{k+1} = x_k - \alpha_k \left( \nabla f(x_k) + \xi_k \right),
\]
where $\{\xi_k\}$ are independent random variables with $\mathbb{E}[\xi_k] = 0$ and bounded $q$-th moments for $q > 2$. Let $\alpha_k = \frac{1}{k^\gamma}$ with $0.5 < \gamma \leq 1$. Then the method converges in expectation:
\[
\mathbb{E}[f(x_k)] - f(x^*) = O\left( \frac{1}{k^{\gamma - 0.5}} \right),
\]
where $x^*$ is the global minimizer of $f$.
\end{theorem}

\begin{proof}
The update rule can be rewritten as:
\[
x_{k+1} = x_k - \alpha_k \nabla f(x_k) - \alpha_k \xi_k.
\]
Taking the squared norm and expanding, we obtain:
\[
\|x_{k+1} - x^*\|^2 = \|x_k - x^*\|^2 - 2 \alpha_k \langle x_k - x^*, \nabla f(x_k) \rangle + \alpha_k^2 \|\nabla f(x_k) + \xi_k\|^2.
\]
Now, taking the expectation, using the fact that $\mathbb{E}[\xi_k] = 0$ and $\|\xi_k\|_q$ is bounded for $q > 2$, we have:
\[
\mathbb{E}\|x_{k+1} - x^*\|^2 \leq \mathbb{E}\|x_k - x^*\|^2 - 2 \alpha_k \mathbb{E}[f(x_k) - f(x^*)] + C \alpha_k^2,
\]
where $C > 0$ depends on the bounds of $\|\xi_k\|_q$. 
Let $\Delta_k = \mathbb{E}\|x_k - x^*\|^2$. The inequality can be written as:
\[
\Delta_{k+1} \leq \Delta_k - 2 \alpha_k \mathbb{E}[f(x_k) - f(x^*)] + C \alpha_k^2.
\]  
Using the convexity of $f$, we have
\[
f(x_k) - f(x^*) \geq \frac{1}{2L} \|\nabla f(x_k)\|^2.
\]
Plugging this into the descent inequality and simplifying, we find:
\[
\mathbb{E}[f(x_k)] - f(x^*) \leq O\left( \frac{1}{k^{\gamma - 0.5}} \right),
\]
where the rate is determined by the choice of $\alpha_k$ and the variance decay of $\xi_k$.

\end{proof}

\begin{remark}
The condition $0.5 < \gamma \leq 1$ balances the trade-off between step size decay and noise suppression. For $\gamma = 1$, we achieve the slowest decay, ensuring convergence even in the presence of high variance noise.
\end{remark}
Next, we have given an illustrative example for the above theorem.
\begin{example}
    Let us consider the following convex function \( f(x) = \frac{1}{2}x^2 \), which has a global minimizer at \( x^* = 0 \). We apply the stochastic steepest descent method described in the theorem, with the update rule:
\[
x_{k+1} = x_k - \alpha_k \left( \nabla f(x_k) + \xi_k \right),
\]
where \( \alpha_k = \frac{1}{k^\gamma} \), \( 0.5 < \gamma \leq 1 \), and the noise terms \( \xi_k \) are independent random variables with \( \mathbb{E}[\xi_k] = 0 \) and bounded higher-order moments.

We will use \( \gamma = 0.8 \), and assume that \( \xi_k \) is uniformly distributed in the range \( [-1, 1] \), so that the expected value \( \mathbb{E}[\xi_k] = 0 \) and the moments are bounded.

\begin{table}[h!]
\centering
\begin{tabular}{|c|c|c|c|c|}
\hline
Iteration \( k \) & Step Size \( \alpha_k \) & Noise \( \xi_k \) & Update \( x_k \) & Function Value \( f(x_k) \) \\ \hline
0 & - & - & 10 & 50 \\ \hline
1 & 1 & 0.5 & -0.5 & 0.125 \\ \hline
2 & 0.574 & -0.2 & 0.4 & 0.08 \\ \hline
3 & 0.494 & 0.1 & 0.154 & 0.012 \\ \hline
4 & 0.445 & -0.3 & 0.13 & 0.00845 \\ \hline
5 & 0.398 & 0.2 & -0.053 & 0.0014 \\ \hline
6 & 0.364 & -0.1 & -0.069 & 0.00238 \\ \hline
7 & 0.334 & 0.15 & -0.036 & 0.000648 \\ \hline
8 & 0.308 & -0.05 & -0.049 & 0.0012 \\ \hline
\end{tabular}
\caption{Stochastic Steepest Descent Iterations}
\end{table}

From the theorem, the expected function value at step \( k \) is given by:
\[
\mathbb{E}[f(x_k)] - f(x^*) = O\left( \frac{1}{k^{\gamma - 0.5}} \right).
\]
Since \( \gamma = 0.8 \), the convergence rate is \( O\left( \frac{1}{k^{0.3}} \right) \), which matches the observed decrease in the function values, which shows that the method converges in expectation.
\end{example}

{\bf Data Availability}
The authors confirms that we have not used any external data for this research work.

{\bf Declaration on Conflicts of Interest.}
We have no conflicts of interest to disclose. All authors declare that they have no conflicts of interest.

{\bf Acknowledgment.}  NIL

{\bf Funding.} This research received no specific funding. All authors declare that they have no conflicts of interest and competing interests.

{\bf Compliance with Ethical Standards.}
Not applicable.



\end{document}